\documentclass{amsart}
\usepackage{amsfonts,amssymb}
\usepackage{enumerate,graphicx}
\usepackage{caption}
\usepackage{subcaption}
\usepackage[square,numbers]{natbib}
\usepackage{mathrsfs,bbm}
\usepackage{tikz,tikzscale,tikz-cd}
\usepackage{multirow,xparse,fp}
\usepackage{environ}
\usepackage{amstext}
\usepackage{hyperref}
\usepackage{diagbox}
\usepackage[pagewise]{lineno}
\usepackage{comment}
\newtheorem{theorem}{Theorem}[section]

\newtheorem{corollary}[theorem]{Corollary}
\newtheorem{lemma}[theorem]{Lemma}
\theoremstyle{definition}

\newtheorem{problem}{Problem}

\numberwithin{equation}{section}
\allowdisplaybreaks[4]

\begin{document}
	\title{Hausdorff dimensions of Beatty multiple shifts}
   \author[Jung-Chao Ban]{Jung-Chao Ban}
	\address[Jung-Chao Ban]{Department of Mathematical Sciences, National Chengchi University, Taipei 11605, Taiwan, ROC.}
	\address{Math. Division, National Center for Theoretical Science, National Taiwan University, Taipei 10617, Taiwan. ROC.}
	\email{jcban@nccu.edu.tw}
	
	\author[Wen-Guei Hu]{Wen-Guei Hu}
	\address[Wen-Guei Hu]{College of Mathematics, Sichuan University, Chengdu, 610064, P. R. China}
	\email{wghu@scu.edu.cn}	
	
	\author[Guan-Yu Lai]{Guan-Yu Lai}
	\address[Guan-Yu Lai]{Department of Mathematical Sciences, National Chengchi University, Taipei 11605, Taiwan, ROC.}
	\email{gylai@nccu.edu.tw}

	
	\thanks{Ban is partially supported by the National Science and Technology Council, ROC (Contract No NSTC 111-2115-M-004-005-MY3) and National Center for Theoretical Sciences. Hu is partially supported by the National Natural Science Foundation of China (Grant No.12271381). Lai is partially supported by the National Science and Technology Council, ROC (Contract NSTC 111-2811-M-004-002-MY2).}
	
	
	
	\begin{abstract}
		   In this paper, the Beatty multiple shift is introduced, which is a generalization of the multiplicative shift of finite type (multiple SFT) [Kenyon, Peres and Solomyak, Ergodic Theory and Dynamical Systems, 2012] and the affine multiple shift [Ban, Hu, Lai and Liao, Advances in Mathematics, 2025]. The Hausdorff and Minkowski dimension formulas are obtained, and the coefficients of the formula is closely related to the classical disjoint covering of the positive integers in number theory. 
	\end{abstract}
	\maketitle
\section{Introduction}\label{sec 1}

Let $2\leq m\in \mathbb{N}$, $\mathcal{A}=\{0, 1, ..., m-1\}$ be the set of symbols and $A$ be an $m\times m$ binary matrix. Motivated by the work of Fan, Liao and Ma \cite{fan2012level} and the multifractal analysis of the multiple averages\footnote{Instead of reviewing the vast and rapidly evolving works and bibliography on the field of multiple averages and multiple ergodic theory, we will direct
the reader to a nice survey and some references for background.} \cite{fan2014some}, Kenyon, Peres and Solomyak \cite{kenyon2012hausdorff} introduced the concept of the \emph{multiplicative shift of finite type} (multiple SFT), that is, 
\begin{equation}
X_{A}^{(q)}=\{x=(x_i)_{i=1}^\infty\in \mathcal{A}^{\mathbb{N}}:A(x_{k},x_{qk})=1\text{
for all }k\in \mathbb{N}\}\text{,}
\end{equation}
and calculate the Hausdorff and Minskowski dimensions of $X_{A}^{(q)}$.

\begin{theorem}[Theorem 1.3, \cite{kenyon2012hausdorff}]
Let $A$ be a $m\times m$ binary matrix which is primitive. 
\begin{enumerate}
\item The Minkowski dimension of $X_{A}^{(q)}$ exists and equals 
\begin{equation}
\dim _{M}X_{A}^{(q)}=(q-1)^{2}\sum_{i=1}^{\infty }\frac{\log _{m}\left\vert
A^{i-1}\right\vert }{q^{i+1}}\text{,}  \label{5}
\end{equation}
where $\left\vert A\right\vert $ stands for the sum of all entries of the matrix $A$.

\item The Hausdorff dimension of $X_{A}^{(q)}$ is presented as 
\begin{equation}
\dim _{H}X_{A}^{(q)}=\frac{q-1}{q}\log _{m}\sum_{i=0}^{m-1}t_{i}\text{,}
\end{equation}
where $(t_{i})_{i=0}^{m-1}$ is the unique positive vector satisfying $t_{i}^{q}=\sum_{j=0}^{m-1}A(i,j)t_{j}$.

\item $\dim_{H}X_{A}^{(q)}=\dim_{M}X_{A}^{(q)}$ if and only if the row sums of $A$ are equal.
\end{enumerate}
\end{theorem}

Previous works related to the dimension results for multiple SFTs include:
(1). dimension theory for multidimensional multiple shift \cite{ban2024hausdorff, brunet2021dimensions, brunet2022contribution}; (2).
dimension spectrum \cite{fan2021multifractal, fan2011multifractal, fan2016multifractal,
peres2014dimensions, peres2012dimension}; (3). topological dynamics for multiple shift \cite{ban2022topologically}; (4). applications to multiple Ising models and the large deviation principle \cite{ carinci2012nonconventional,
chatterjee2016nonlinear,chazottes2014thermodynamic,kifer2015lectures,kifer2014nonconventionallarge}, and (5). dimension theory for $3$-multiple shifts \cite{ban2019pattern}.

Later, Ban, Hu, Lai and Liao \cite{ban2024hausdorffaffine} introduced the concept of the \emph{affine
multiple SFT} which is defined as
\[
X_{A}^{\left\langle p,a,q,b\right\rangle }=\left\{ x=(x_i)_{i=1}^\infty\in \mathcal{A}^{\mathbb{N}}:A(x_{pk+a},x_{qk+b})=1\text{ for all }k\in \mathbb{N}\right\} 
\text{.} 
\]
Selecting $p=1$ and $a=b=0$, it is evident that $X_{A}^{(q)}$ is a special case of $X_{A}^{\left\langle p,a,q,b\right\rangle }$.  The Hausdorff and Minkowski dimension formulas have been established therein.

\begin{theorem}[Minkowski dimension, \cite{ban2024hausdorffaffine}]\label{Thm: 4}
Let $A$ be an $m\times m$ irreducible binary matrix, $1\leq p<q$ and $a,b\in \mathbb{Z}$.
\begin{enumerate}
\item If $(p,q)\nmid (b-a)$, then 
\[
\dim _{M}X_{A}^{\left\langle p,a,q,b\right\rangle }=1-\frac{2}{q}+\frac{1}{q}\log_{m}\left\vert A\right\vert,
\]
where $(p,q)$ denotes the greatest common divisor of $p$ and $q$.

\item If $(p,q)\mid (b-a)$, then 
\[
\dim _{M}X_{A}^{\left\langle p,a,q,b\right\rangle }=1-\frac{2q_{1}-1}{(p,q)q_{1}^{2}}+\frac{(q_{1}-1)^{2}}{(p,q)}\sum_{i=2}^{\infty }\frac{\log
_{m}\left\vert A^{i-1}\right\vert }{q_{1}^{i+1}}\text{,} 
\]
where $q=q_{1}(p,q)$. In particular, if $(p,q)=1$, we have 
\begin{equation*}
\dim _{M}X_{A}^{\left\langle p,a,q,b\right\rangle
}=(q-1)^{2}\sum_{i=1}^{\infty }\frac{\log_{m}\left\vert A^{i-1}\right\vert 
}{q^{i+1}}\text{.}
\end{equation*}
\end{enumerate}
\end{theorem}

\begin{theorem}[Hausdorff dimension, \cite{ban2024hausdorffaffine}]
\label{Thm: 3}Let $A$ be an $m\times m$ irreducible
binary matrix. 
\begin{enumerate}
\item If $1\leq p<q$, $a, b\in \mathbb{Z}$ and $(p,q)\nmid (b-a)$, then 
\[
\dim _{H}X_{A}^{\left\langle p,a,q,b\right\rangle }=1-\frac{1}{p}-\frac{1}{q}+\frac{1}{p}\log _{m}\sum_{i=0}^{m-1}a_{i}{}^{\frac{p}{q}}, 
\]
where $a_i=\sum_{j=0}^{m-1}A(i,j)$.

\item If $1<p<q$, $a, b\in \mathbb{Z}$, $(p,q)\mid (b-a)$ and $p_{1}>1$, then 
\[
\dim _{H}X_{A}^{\left\langle p,a,q,b\right\rangle }=1-\frac{1}{p}-\frac{1}{q}+\frac{1}{(p,q)p_{1}q_{1}}+\sum_{i=2}^{\infty }\left(
\sum_{j=1}^{i}P_{ij}\right) \log _{m}t_{\emptyset ;i}\text{,} 
\]
where $p=p_{1}(p,q)$, $P_{ij}$ and $t_{\emptyset ;i}$ are defined in \cite[(3.4) and (3.5)]{ban2024hausdorffaffine}. 

\item If $1<p<q$, $a,b\in \mathbb{Z}$, $(p,q)|(b-a)$, $p_{1}=1$ and $A$ is primitive, then 
\[
\dim _{H}X_{A}^{\left\langle p,a,q,b\right\rangle }=1-\frac{1}{p}+\frac{q_{1}-1}{(p,q)q_{1}}\log_{m}\sum_{i=0}^{m-1}t_{i}\text{,} 
\]
where $\left( t_{i}\right) _{i=0}^{m-1}$ is the unique positive vector
satisfying $t_{i}^{q_1}=\sum_{j=0}^{m-1}A(i,j)t_{j}$.

\item If $1=p<q$, $a,b\in \mathbb{Z}$ and $A$ is primitive, then 
\begin{equation}
\dim _{H}X_{A}^{\left\langle p,a,q,b\right\rangle }=\frac{q-1}{q}\log_{m}\sum_{i=0}^{m-1}t_{i}\text{,}
\end{equation}
where $(t_{i})_{i=0}^{m-1}$ is the unique positive vector satisfying $t_{i}^{q}=\sum_{j=0}^{m-1}A(i,j)t_{j}$.

\item $\dim _{H}X_{A}^{\left\langle p,a,q,b\right\rangle }=\dim
_{M}X_{A}^{\left\langle p,a,q,b\right\rangle }$ if and only if the row sums of $A$ are equal.
\end{enumerate}
\end{theorem}

The purpose of this article is to introduce a more general class of multiple shifts. Let $\alpha, \beta, \gamma, \delta \in \mathbb{R}$ and $1\leq \alpha <\gamma $, the \emph{Beatty multiple shift} with respect to the tuple $\left( \alpha
,\beta ,\gamma ,\delta \right)$ is defined by
\[
X_{A}^{\left[ \alpha ,\beta ,\gamma ,\delta \right] }=\left\{
x=(x_{i})_{i=1}^{\infty }\in \mathcal{A}^{\mathbb{N}}:A(x_{\left\lfloor
\alpha k+\beta \right\rfloor },x_{\left\lfloor \gamma k+\delta \right\rfloor
})=1\text{ for all }k\in \mathbb{N}\right\},
\]
where $\left\lfloor \tau \right\rfloor $ denotes the floor function of $\tau$. The term Beatty multiple shift is derived from the notion of the Beatty sequence $\{\left\lfloor \xi n+\eta \right\rfloor \}_{n=1}^\infty$, which is a crucial method for studying the primes in the field of analytic number theory \cite{kuipers2012uniform}. It can be easily seen that the shift $X_{A}^{\left\langle p,a,q,b\right\rangle }$ is a special case of $X_{A}^{\left[ \alpha ,\beta ,\gamma ,\delta \right] }$ if $(\alpha ,\beta ,\gamma,\delta )=(p,a,q,b)$. However, most types of $X_{A}^{\left[\alpha, \beta, \gamma, \delta \right] }$ cannot be categorized as a
type of $X_{A}^{\left\langle p,a,q,b\right\rangle }$. Here, we provide the Hausdorff and Minkowski dimension formulas for $X_{A}^{\left[ \alpha ,\beta ,\gamma ,\delta \right] }$ and determine when these two dimensions coincide. To give a complete presentation of the dimension formula, we provide technical notations before presenting our main result.

Define a function 
\[f:\{ \left\lfloor \alpha k+\beta \right\rfloor :k\in \mathbb{N}\}\to \{ \left\lfloor \gamma k+\delta \right\rfloor :k\in \mathbb{N}\}
\]
by 
\begin{equation}\label{iteration}
  f(\left\lfloor \alpha k+\beta \right\rfloor )=\left\lfloor \gamma k+\delta
\right\rfloor.  
\end{equation}
Note that $f$ is well-defined by Lemma \ref{lma1}. 

Denote the set of positive integers of the form $\left\lfloor \alpha k+\beta \right\rfloor$ for some $k\in \mathbb{N}$ by 
\[
S(\alpha ,\beta ):=\left\{ \left\lfloor \alpha k+\beta \right\rfloor :k\in 
\mathbb{N}\right\} \cap \mathbb{N}.
\]
Let 
\begin{equation*}
A_{1} =\mathbb{N}\setminus \left[ S(\alpha ,\beta )\cup S(\gamma ,\delta )\right]
\end{equation*}
be the set of positive integers that are not of the form $\left\lfloor \alpha k+\beta \right\rfloor$ or $\left\lfloor \gamma k+\delta \right\rfloor$ for any $k\in \mathbb{N}$, and let
\begin{equation*}
 A_{2} =\left\{ x\in S(\alpha ,\beta )\setminus S(\gamma ,\delta ):f(x)\in
S(\gamma ,\delta )\setminus S(\alpha ,\beta )\right\} 
\end{equation*}
be the set of elements in $S(\alpha ,\beta )\setminus S(\gamma ,\delta )$ that can undergo exactly one iteration under $f$. For $i\geq 3$, we define the set of elements in $S(\alpha ,\beta )\setminus S(\gamma ,\delta )$ that can undergo exactly $i-1$ iterations under $f$ by
\begin{align*}
A_{i} =&\{x\in S(\alpha ,\beta )\setminus S(\gamma ,\delta ):f^{j}(x)\in
S(\alpha ,\beta )\cap S(\gamma ,\delta ) \\
&\text{ for all }1 \leq j\leq i-2\text{ and }f^{i-1}(x)\in S(\gamma ,\delta)\setminus S(\alpha ,\beta )\}\text{,}
\end{align*}
and define the set of elements in $S(\alpha ,\beta )\setminus S(\gamma ,\delta )$ that can undergo infinitely many iterations under $f$ by
\[
A_{\infty }=\left\{ x\in S(\alpha ,\beta )\setminus S(\gamma ,\delta
):f^{j}(x)\in S(\alpha ,\beta )\cap S(\gamma ,\delta )\text{ for all }j\geq
1\right\} \text{.} 
\]
For $i\in \mathbb{N}\cup \{\infty \},$ we define the \emph{density} of $A_i$ by
\begin{equation}
d_{i}=\lim_{\left\vert n_{1}-n_{2}\right\vert \rightarrow \infty }\frac{\left\vert A_{i}\cap \lbrack n_{1},n_{2}]\right\vert }{\left\vert
n_{1}-n_{2}\right\vert }\text{,}  \label{1}
\end{equation}
whenever the limit (\ref{1}) exists. It is worth noting that if (\ref{1}) exists, then for $1\leq j\leq i\in \mathbb{N}\cup \{\infty \}$, the following limit exists as well. 
\[
d_{i,j}=\lim_{n\rightarrow \infty }\frac{\left\vert A_{i,j}(n)\right\vert }{n}\text{ and }d_{i}=\sum_{j=1}^{i}d_{i,j}\text{,} 
\]
where 
\[
A_{i,j}(n):=\left\{ x\in A_{i}:\left\vert \{x,\ldots ,f^{i-1}(x)\}\cap
\lbrack 1,n]\right\vert =j\right\}
\]
is the set of elements in $A_i$ such that exactly $j$ elements of $\{x,\ldots ,f^{i-1}(x)\}$ lie in the interval $[1,n]$.

Using the notations and concepts mentioned above, we present the formulas for the Minkowski and Hausdorff dimensions of the Beatty multiple shift $
X_{A}^{\left[ \alpha ,\beta ,\gamma ,\delta \right] }$.

\begin{theorem}\label{Thm: 1}
Let $A$ be an $m\times m$ irreducible binary matrix, $1\leq \alpha <\gamma\in \mathbb{R}$ and $\beta, \delta \in\mathbb{R}$. If $d_{i}$ exists for all $i\in \mathbb{N}\cup \{\infty \}$, then we have the following assertions.

\begin{enumerate}
\item The Minkowski dimension of $X_{A}^{\left[ \alpha ,\beta ,\gamma
,\delta \right] }$ is presented as follows. 
\begin{eqnarray*}
\dim _{M}X_{A}^{\left[ \alpha ,\beta ,\gamma ,\delta \right] }
&=&\sum_{i=1}^{\infty }\left\{ \frac{1}{\left( \frac{\gamma }{\alpha }\right) ^{i-1}}d_{i}+\left[ \frac{1}{\left( \frac{\gamma }{\alpha }\right)
^{i-1}}-\frac{1}{\left( \frac{\gamma }{\alpha }\right) ^{i}}\right] \left(
\sum_{j=i+1}^{\infty }d_{i}+d_{\infty }\right) \right\} \\
&&\times \log_{m}\left\vert A^{i-1}\right\vert \text{.}
\end{eqnarray*}

\item If $A$ is primitive, then 
\[
\dim _{H}X_{A}^{\left[ \alpha ,\beta ,\gamma ,\delta \right]
}=d_{1}+\sum_{i=2}^{\infty }d_{i}\log _{m}t_{\emptyset .i}+d_{\infty }\log
_{m}\sum_{i=0}^{m-1}t_{i}\text{,} 
\]
where $t_{\emptyset ,i}$ are defined in (\ref{eq 3.1-2}), and $(t_{i})_{i=0}^{m-1}$ is the
unique positive vector satisfying $t_{i}^{\frac{\gamma }{\alpha }}=\sum_{j=0}^{m-1}A(i,j)t_{j}$.

\item $\dim _{H}X_{A}^{\left[ \alpha ,\beta ,\gamma ,\delta \right] }=\dim_{M}X_{A}^{\left[ \alpha ,\beta ,\gamma ,\delta \right] }$ if and only if the row sums of $A$ are equal.
\end{enumerate}
\end{theorem}
We remark that if $d_\infty=0$, then the assumption of $A$ in (2) of Theorem \ref{Thm: 1} can be replaced by ``$A$ is irreducible''. In accordance with Theorem \ref{Thm: 1}, the sequence 
\begin{equation} \label{2}
\mathbf{d}=(d_{1},d_{2},d_{3},\ldots ,d_{\infty }) 
\end{equation}
is the only factor that determines the explicit formulas of the Hausdorff and Minkowski dimensions of $X_{A}^{\left[ \alpha ,\beta ,\gamma ,\delta\right]}$. The calculation of vector $\mathbf{d}$ is a challenge that relies on many previous works on the Beatty sequences \cite{graham1973covering, harman2015primes, harman2016primes,kuipers2012uniform}. The tuple $(\alpha ,\beta ,\gamma ,\delta )\in \mathbb{R}^{4}$ with $1\leq \alpha <\gamma $ is classified in Figure \ref{fig1}. \begin{figure}[]
\centering
\includegraphics[width=10cm]{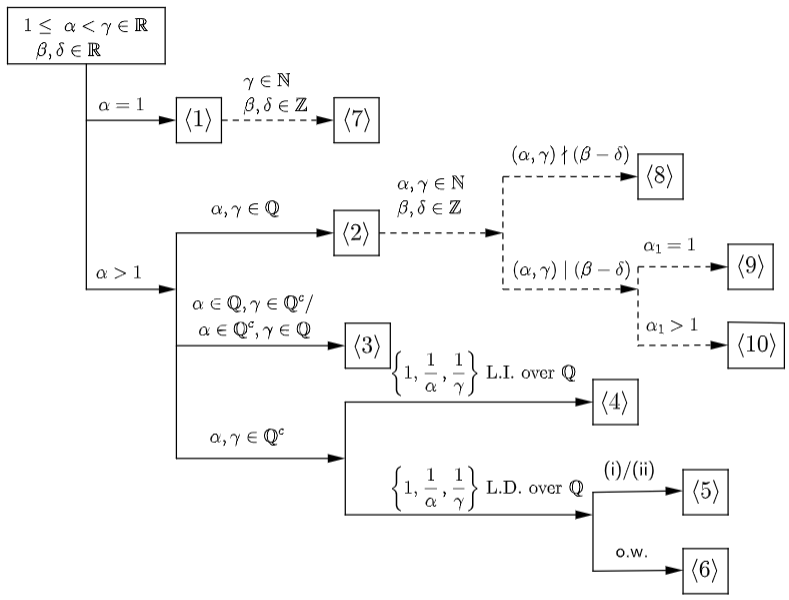}
\caption{(i) $\frac{n}{\alpha}+\frac{m}{\gamma}=1$ and $\frac{n\beta}{\alpha}+\frac{m\delta}{\gamma}\in \mathbb{Z}$, and (ii) $\frac{n}{\alpha}-\frac{m}{\gamma}=0$ with $(n,m)=1$ and $ 1-\frac{m}{\alpha}\geq \left\{\frac{m(\beta-\delta)}{\alpha}\right\}\geq \frac{m}{\alpha}$. The dashed line means the cases when it satisfies the conditions under such line.}
    \label{fig1}
\end{figure}
Thus,
\[
\left\{ (\alpha ,\beta ,\gamma ,\delta )\in \mathbb{R}^{4}:1\leq \alpha
<\gamma \right\} =\bigcup_{i=1}^{6}\left\langle i\right\rangle .
\]
For each region, we provide the explicit expression of $\mathbf{d}$ as in Theorem \ref{Thm: 2}. If there is no ambiguity, we write $\mathbf{d}=(d_{1},d_{2},\underline{d},d_{\infty })$, where $\underline{d}=\left(d_{i}\right)_{i\geq 3}$.

\begin{theorem}\label{Thm: 2}
The following assertions hold true.
\begin{enumerate}
\item If $(\alpha ,\beta ,\gamma ,\delta )\in \left\langle 1\right\rangle $, then $\mathbf{d}=(0,0,\underline{0},1-\frac{1}{\gamma })$, where $\underline{0}=(0)_{i\geq 3}$.

\item If $(\alpha ,\beta ,\gamma ,\delta )\in \left\langle 2\right\rangle $,
then $\mathbf{d}$ is defined in Section \ref{appendix}.

\item If $(\alpha ,\beta ,\gamma ,\delta )\in \left\langle 3\right\rangle
\cup \left\langle 4\right\rangle $, then 
\[
\mathbf{d}=\left( \frac{\left( \alpha -1\right) \left( \gamma -1\right) }{\alpha \gamma },\frac{\left( \alpha -1\right) \left( \gamma -1\right) }{\alpha ^{2}\gamma },\underline{d},0\right), 
\]
where 
\[
\underline{d}=\left( \frac{\left( \alpha -1\right) \left( \gamma -1\right) }{\alpha ^{i}\gamma }\right)_{i\geq 3}. 
\]

\item If $(\alpha ,\beta ,\gamma ,\delta )\in \left\langle 5\right\rangle $,
then $\mathbf{d}=(1-\frac{1}{\alpha }-\frac{1}{\gamma },\frac{1}{\alpha },\underline{0},0)$.
\end{enumerate}
\end{theorem}

Note that the cases $\left\langle 7\right\rangle, \left\langle 8\right\rangle,\left\langle 9\right\rangle,\left\langle 10\right\rangle$ form a classification of the set $\{(p,a,q,b)\in \mathbb{Z}^4 :  1\leq p <q\}$, i.e.,
\begin{equation*}
    \{(p,a,q,b)\in \mathbb{Z}^4 :  1\leq p <q\}=\bigcup_{i=7}^{10} \left\langle i\right\rangle.
\end{equation*}
In fact, the Hausdorff dimension of $X_A^{[p,a,q,b]}$ with $(p,a,q,b)\in \left\langle 7\right\rangle$ (resp. $\left\langle 8\right\rangle, \left\langle 9\right\rangle, \left\langle 10\right\rangle$) is obtained by Theorem \ref{Thm: 3} (4) (resp. (1), (3), (2)), and the Minkowski dimension of $X_A^{[p,a,q,b]}$ with $(p,a,q,b)\in \left\langle 7\right\rangle$ (resp. $\left\langle 8\right\rangle, \left\langle 9\right\rangle, \left\langle 10\right\rangle$) is obtained by Theorem \ref{Thm: 4} (2) (resp. (1), (2), (2)). 
\begin{corollary}\label{Cor: 1}
The following assertions hold true.

\begin{enumerate}
\item If $(\alpha ,\beta ,\gamma ,\delta )\in \left\langle 7\right\rangle $, then $\mathbf{d}=(0,0,\underline{0},1-\frac{1}{\gamma })$.

\item If $(\alpha ,\beta ,\gamma ,\delta )\in \left\langle 8\right\rangle $, then $\mathbf{d}=(1-\frac{1}{\alpha }-\frac{1}{\gamma },\frac{1}{\alpha },\underline{0},0)$.

\item If $(\alpha ,\beta ,\gamma ,\delta )\in \left\langle 9\right\rangle $,
then 
\[
\mathbf{d}=\left( 1-\frac{1}{\alpha }-\frac{1}{\gamma }+\frac{1}{(\alpha
,\gamma )\alpha _{1}\gamma _{1}},0,\underline{0},\frac{\gamma _{1}-1}{\alpha
\gamma _{1}}\right) \text{,} 
\]
where $\alpha=\alpha_1(\alpha ,\gamma )$ and $\gamma=\gamma_1 (\alpha ,\gamma )$.

\item If $(\alpha ,\beta ,\gamma ,\delta )\in \left\langle 10\right\rangle $, then 
\[
\mathbf{d}=\left( 1-\frac{1}{\alpha }-\frac{1}{\gamma }+\frac{1}{(\alpha
,\gamma )\alpha _{1}\gamma _{1}},\frac{(\gamma _{1}-1)(\alpha _{1}-1)}{%
\gamma _{1}\alpha \alpha _{1}},\underline{d},0\right) \text{,} 
\]
where 
\[
\underline{d}=\left( \frac{(\gamma _{1}-1)(\alpha_{1}-1)}{\gamma_{1}\alpha\alpha_{1}^{i-1}}\right)_{i\geq 3}
\text{.} 
\]
\end{enumerate}
\end{corollary}
The Table \ref{tab: 1} provides the relation (of ${\bf d}$) between Theorem \ref{Thm: 2} and Corollary \ref{Cor: 1}. \begin{table}[]
    \centering
    \begin{tabular}{ |l|cccc|c|c|ccc|c| }
 \hline
 &$\langle 8\rangle$&$\langle 10\rangle$&$\langle 9\rangle$&$\langle 7\rangle$&$\langle 1\rangle$&$\langle 2\rangle$&$\langle 4\rangle$&$\langle 5\rangle$&$\langle 6\rangle$&$\langle 3\rangle$\\
 \hline
 $d_1$  & $\zeta$    &$\theta$&  $\theta$& 0  &0&$\xi$&$\mu$&$\zeta$&?&$\mu$\\
 $d_2$&   $\eta$  & $\iota_2$   &0& 0  &0&$\rho_2$&$\nu_2$&$\eta$&?&$\nu_2$\\
 $\underline{d}$&$\underline{0}$& $(\iota_i)_{i\geq 3}$&$\underline{0}$&$\underline{0}$&$\underline{0}$ &$(\rho_i)_{i\geq 3}$&$(\nu_i)_{i\geq 3}$&$\underline{0}$&?&$(\nu_i)_{i\geq 3}$\\
 $d_\infty$ &0& 0& $\kappa$ & $\lambda$&$\lambda$&I&0&0&?&0\\
  \hline
 Ref.& \multicolumn{4}{c|}{[Theorem \ref{Thm: 3}]}  & \multicolumn{6}{c|}{This paper} \\
 \hline
\end{tabular}
    \caption{$\zeta=1-\frac{1}{\alpha}-\frac{1}{\gamma}$, $\eta=\frac{1}{\alpha}$, $\theta=1-\frac{1}{\alpha}-\frac{1}{\gamma}+\frac{1}{(\alpha,\gamma)\alpha_1 \gamma_1}$, $\iota_i=\frac{(\gamma_1-1)(\alpha_1-1)}{\gamma_1\alpha\alpha_1^{i-1}}$, $\kappa=\frac{1}{\alpha}\frac{\gamma_1-1}{\gamma_1}$, $\lambda=1-\frac{1}{\gamma}$, $\mu=\frac{(\alpha-1)(\gamma-1)}{\alpha\gamma}$, and $\nu_i= \frac{(\gamma-1)(\alpha-1)}{\gamma\alpha^i}$, and $\xi, \rho_i$ are defined in Section \ref{appendix}.}
    \label{tab: 1}
\end{table}
It is easy to verify that the Hausdorff and Minkowski dimensions obtained
from Corollary \ref{Cor: 1} and Theorem \ref{Thm: 1} coincide with the formulas in Theorem \ref{Thm: 3} and Theorem \ref{Thm: 4}. Thus, it is
evident that Theorem \ref{Thm: 1} and Theorem \ref{Thm: 2} have a more general nature than Theorem \ref{Thm: 3} and Theorem \ref{Thm: 4}. It is unfortunate that we cannot solve all cases. Region $\left\langle 6\right\rangle$ remains as an open problem.

\begin{problem}
How to determine the infinite sequence $\mathbf{d}=(d_{1},d_{2},d_{3},\ldots,d_{\infty })$ for $(\alpha ,\beta ,\gamma ,\delta )\in \left\langle
6\right\rangle $?
\end{problem}

We end this section by making two remarks. First, the numbers $\alpha ,\beta,\gamma $ and $\delta $ in $X_{A}^{\left[ \alpha ,\beta ,\gamma ,\delta\right] }$ can be selected as real numbers, which indicates that the class of $X_{A}^{\left[ \alpha ,\beta ,\gamma ,\delta \right] }$ is significantly larger than that of $X_{A}^{\left\langle p,a,q,b\right\rangle }$. Second, the Hausdorff and Minkowski dimensions are calculated using the infinite sequence $\mathbf{d}$, which is also closely associated with the distribution property of the Beatty sequences. This bridge ties together the dimension theory of multiple SFTs and the research on Beatty sequences in number theory. 

\section{Proof of Theorems}
In this section, we give the proofs of Theorems \ref{Thm: 1} and \ref{Thm: 2}. The following lemma provides that $f$ in (\ref{iteration}) is well-defined.
\begin{lemma}\label{lma1}
    Let $1\leq\gamma\in\mathbb{R}$, $\delta\in\mathbb{R}$ and $k_1,k_2\in\mathbb{N}$. Then, $\left\lfloor \gamma k_1+\delta\right\rfloor= \left\lfloor \gamma k_2+\delta\right\rfloor$ if and only if $k_1= k_2$. 
\end{lemma}

\begin{proof}
    If $k_1=k_2$, then it is clear to see that $ \gamma k_1+\delta= \gamma k_2+\delta$ and so $\left\lfloor \gamma k_1+\delta\right\rfloor= \left\lfloor \gamma k_2+\delta\right\rfloor$. Conversely, if $\left\lfloor \gamma k_1+\delta\right\rfloor= \left\lfloor \gamma k_2+\delta\right\rfloor$ with $k_1< k_2$ (similar for the case $k_1>k_2$), then by $\gamma\geq 1$, we have
    \begin{align*}
        &\left\lfloor \gamma k_1+\delta\right\rfloor= \left\lfloor \gamma k_2+\delta\right\rfloor=\left\lfloor \gamma k_1+\gamma(k_2-k_1)+\delta\right\rfloor\\
        \geq&  \left\lfloor \gamma k_1+\gamma+\delta\right\rfloor\geq\left\lfloor \gamma k_1+1+\delta\right\rfloor=\left\lfloor \gamma k_1+\delta\right\rfloor+1,
    \end{align*}
    which implies $0\geq 1$ and thus gives a contradiction. The proof is complete.
\end{proof}

\begin{proof}[Proof of Theorem \ref{Thm: 1}]     
\item[\bf 1.] Note that 
\begin{equation}\label{eq pf2}
    \mathbb{N}=\left(\bigsqcup_{i=1}^\infty \bigsqcup_{x\in A_i}\{x, ..., f^{i-1}(x)\}\right) \sqcup \bigsqcup_{x\in A_\infty}\{x, f(x), ...\}\sqcup R,
\end{equation}
where $R$ is the remaining elements in $\mathbb{N}$. In fact, it can be verified that \[
\lim_{|n_1-n_2|\to\infty}\frac{|R\cap [n_1,n_2]|}{|n_1-n_2|}=0
\] by using (\ref{eq est}) with the fact $|\left\{ \left\lfloor \tau k+\eta \right\rfloor :k\in 
\mathbb{N}\right\} \setminus \mathbb{N}|<\infty$ for $\tau \geq 1$ and $\eta \in\mathbb{R}$. In order to simplify the proof, in the following, we omit $R$ in the decomposition (\ref{eq pf2}) of $\mathbb{N}$.
     
Now, we estimate the upper and lower bounds of $f^{i-1}(x)$ with respect to $x$, where $x\in A_i$ and $ i\in \mathbb{N}\cup\{\infty\}$. For $ i\in \mathbb{N}\cup\{\infty\}$ and $x\in A_i$, we have
    \begin{equation*}
        x=\left\lfloor \alpha k+\beta \right\rfloor\mbox{ for some }k\in \mathbb{N}.
    \end{equation*}
Then,
\begin{equation*}
    x\leq \alpha k+\beta <x+1.
\end{equation*}
Thus,
\begin{equation*}
   \frac{x-\beta}{\alpha} \leq k <\frac{x+1-\beta}{\alpha}.
\end{equation*}
Then,
\begin{equation*}
   \frac{\gamma (x-\beta)}{\alpha}+\delta-1\leq \gamma k+\delta -1<  \left\lfloor \gamma k+\delta \right\rfloor\leq \gamma k+\delta <\frac{\gamma (x+1-\beta)}{\alpha}+\delta.
\end{equation*}
Hence,
\begin{equation}\label{eq pf1}
    \frac{\gamma (x-\beta)}{\alpha}+\delta-1 <f(x)<\frac{\gamma (x+1-\beta)}{\alpha}+\delta.
\end{equation}
Then, by (\ref{eq pf1}), we have that for $0\leq \ell\leq i-1$, 
    \begin{align*}
 f^\ell(x)<&\left(\frac{\gamma}{\alpha}\right)^\ell \left(x+(1-\beta)\sum_{i=1}^{\ell} \frac{1}{\left(\frac{\gamma}{\alpha}\right)^{i-1}}+\delta\sum_{i=1}^{\ell} \frac{1}{\left(\frac{\gamma}{\alpha}\right)^{i}}\right)\\
   \leq &\left(\frac{\gamma}{\alpha}\right)^\ell \left[x+(1+|\beta|+|\delta|)\sum_{i=0}^{\ell} \frac{1}{\left(\frac{\gamma}{\alpha}\right)^{i}}\right]\\
   \leq &\left(\frac{\gamma}{\alpha}\right)^\ell \left[x+(1+|\beta|+|\delta|)\sum_{i=0}^{\infty} \frac{1}{\left(\frac{\gamma}{\alpha}\right)^{i}}\right]\\
   =&\left(\frac{\gamma}{\alpha}\right)^\ell \left[x+(1+|\beta|+|\delta|)\frac{\alpha}{\gamma-\alpha}\right],
\end{align*}
and
  \begin{align*}
  f^\ell(x)>&\left(\frac{\gamma}{\alpha}\right)^\ell \left(x-\beta\sum_{i=1}^{\ell} \frac{1}{\left(\frac{\gamma}{\alpha}\right)^{i-1}}+(\delta-1)\sum_{i=1}^{\ell} \frac{1}{\left(\frac{\gamma}{\alpha}\right)^{i}}\right)\\
  \geq &\left(\frac{\gamma}{\alpha}\right)^\ell \left[x-(1+|\beta|+|\delta|)\sum_{i=0}^{\ell} \frac{1}{\left(\frac{\gamma}{\alpha}\right)^{i}}\right]\\
  =&\left(\frac{\gamma}{\alpha}\right)^\ell \left[x-(1+|\beta|+|\delta|)\sum_{i=0}^{\infty} \frac{1}{\left(\frac{\gamma}{\alpha}\right)^{i}}\right]\\
  =&\left(\frac{\gamma}{\alpha}\right)^\ell \left[x-(1+|\beta|+|\delta|)\frac{\alpha}{\gamma-\alpha}\right].
\end{align*}
Thus,
\begin{equation}\label{eq est}
\left(\frac{\gamma}{\alpha}\right)^\ell \left(x-C\right)< f^\ell(x)<\left(\frac{\gamma}{\alpha}\right)^\ell \left(x+C\right),
\end{equation}
where $C=(1+|\beta|+|\delta|)\frac{\alpha}{\gamma-\alpha}$.

For $n\geq 1$, by (\ref{eq pf2}), we have
\begin{equation}\label{eq pf3}
    [1,n]= \left(\bigsqcup_{i=1}^\infty \bigsqcup_{x\in A_i}\{x, ..., f^{i-1}(x)\}\cap [1,n]\right) \sqcup \bigsqcup_{x\in A_\infty}\{x, f(x), ...\}\cap [1,n].
\end{equation}
By (\ref{eq est}), we have
\begin{equation}\label{eq pf4}
\begin{aligned}
     \bigsqcup_{x\in A_i}\{x, ..., f^{i-1}(x)\}\cap [1,n]=& \bigsqcup_{x\in A_i\cap E_i(n)}\{x, ..., f^{i-1}(x)\}\\
     &\sqcup \bigsqcup_{\ell=0}^{i-2}\bigsqcup_{x\in A_i\cap F_\ell(n)}\{x, ..., f^{\ell}(x)\},
\end{aligned}
\end{equation}
where $E_i(n)=\left[1, \frac{n\pm C}{\left(\frac{\gamma}{\alpha}\right)^{i-1}}\right]$ and $F_\ell(n)=\left(\frac{n\pm C}{\left(\frac{\gamma}{\alpha}\right)^{\ell+1}}, \frac{n\pm C}{\left(\frac{\gamma}{\alpha}\right)^{\ell}}\right]$.

Then, by (\ref{eq pf3}) and (\ref{eq pf4}), we have
\begin{equation}\label{eq pf5}
  \begin{aligned}  \left|\mathcal{P}\left(X_A^{\alpha,\gamma;\beta,\delta}, [1, n]\right)\right|
  =&\prod_{i=1}^\infty \left\{\prod_{x\in A_i\cap E_i(n)}\left|\mathcal{P}(X_A^{\alpha,\gamma;\beta,\delta} , \{x, ..., f^{i-1}(x)\})\right |\right.\\
  &\left.\times \prod_{\ell=0}^{i-2}\prod_{x\in A_i\cap F_\ell(n)}\left|\mathcal{P}\left(X_A^{\alpha,\gamma;\beta,\delta}, \{x, ..., f^{\ell}(x)\}\right)\right| \right\}\\
  &\times\prod_{\ell=0}^{\infty}\prod_{x\in A_\infty\cap F_\ell(n)}\left|\mathcal{P}\left(X_A^{\alpha,\gamma;\beta,\delta}, \{x, ..., f^{\ell}(x)\}\right)\right|.
  \end{aligned}
\end{equation}
Since $\left|\mathcal{P}\left(X_A^{\alpha,\gamma;\beta,\delta}, \{x, ..., f^{\ell}(x)\}\right)\right|=\left|A^\ell \right|$, by (\ref{eq pf5}), we have
  \begin{align*}  \left|\mathcal{P}\left(X_A^{\alpha,\gamma;\beta,\delta}, [1, n]\right)\right|
  =&\prod_{i=1}^\infty \left\{ \left|A^{i-1}\right|^{\left| A_i\cap E_i(n)\right|}\times \prod_{\ell=0}^{i-2} \left|A^\ell\right|^{\left| A_i\cap F_\ell(n)\right|} \right\}\\
&\times\prod_{\ell=0}^{\infty}\left|A^{\ell}\right|^{\left|A_\infty\cap F_\ell(n)\right|}.
  \end{align*}
Hence,
\begin{equation}\label{eq pf6}
    \begin{aligned}
&\frac{\log_m\left|\mathcal{P}\left(X_A^{\alpha,\gamma;\beta,\delta}, [1, n]\right)\right|}{n}\\ =&\sum_{i=1}^\infty \left\{\frac{\left| A_i\cap E_i(n)\right|}{n} \log_m\left|A^{i-1}\right|+ \sum_{\ell=0}^{i-2} \frac{\left| A_i\cap F_\ell(n)\right|}{n}\log_m\left|A^\ell\right| \right\}\\
  &+\sum_{\ell=0}^{\infty}\frac{\left|A_\infty\cap F_\ell(n)\right|}{n}\log_m\left|A^{\ell}\right|.
  \end{aligned}
\end{equation}
By (\ref{1}), we have 
\begin{equation}\label{eq pf7}
    \begin{aligned}
        \lim_{n\to\infty}\frac{\left| A_i\cap E_i(n)\right|}{n} =&\frac{1}{(\frac{\gamma}{\alpha})^{i-1}}d_i,\\
        \lim_{n\to\infty}\frac{\left| A_i\cap F_\ell(n)\right|}{n}=&\left[\frac{1}{(\frac{\gamma}{\alpha})^{\ell}}-\frac{1}{(\frac{\gamma}{\alpha})^{\ell+1}}\right]d_i,\\
        \lim_{n\to\infty}\frac{\left|A_\infty\cap F_\ell(n)\right|}{n}=&\left[\frac{1}{(\frac{\gamma}{\alpha})^{\ell}}-\frac{1}{(\frac{\gamma}{\alpha})^{\ell+1}}\right]d_\infty.
    \end{aligned}
\end{equation}
Then, by (\ref{eq pf6}) and (\ref{eq pf7}), we obtain 
   \begin{align*}
    \dim_M X^{[\alpha,\beta,\gamma,\delta]}_A =&\sum_{i=1}^\infty\left\{\frac{1}{(\frac{\gamma}{\alpha})^{i-1}}d_i+\left[\frac{1}{(\frac{\gamma}{\alpha})^{i-1}}-\frac{1}{(\frac{\gamma}{\alpha})^{i}}\right]\left(\sum_{j=i+1}^\infty d_j+d_{\infty}\right)\right\}\\
    &\times \log_m|A^{i-1}|.
    \end{align*}
    The proof is complete.
\item[\bf 2.]
For $i\in\mathbb{N}$, let $\mu_i$ be the probability measure defined in \cite[Theorem 1.3 (2)]{ban2024hausdorffaffine} with $P_{i,j}$ replaced by $d_{i,j}$. And let $\mu_\infty$ be the probability measure defined in \cite[Theorem 1.3 (2)]{ban2024hausdorffaffine} with $q_1$ replaced by $\frac{\gamma}{\alpha}$ (also see \cite[Corollary 2.6]{kenyon2012hausdorff} with $q$ replaced by $\frac{\gamma}{\alpha}$). Define a probability measure on $X^{[\alpha,\beta,\gamma,\delta]}_A$ by defining the measure on cylinder sets in $\alpha_n$ of $X^{[\alpha,\beta,\gamma,\delta]}_A$. That is, for any $[x_1,...,x_n]\in \alpha_n$,
    \begin{align*}
        \mathbb{P}_{\infty}([x_1\cdots x_n])&:=\prod_{i=1}^\infty\prod_{x\in A_i}\mu_i([x_{U}])\times \prod_{x\in A_\infty} \mu_\infty ([x_U]),
    \end{align*}
where $U=\{x, f(x), ...\}\cap [1, n]$, $\mu_1([i_1])=\frac{1}{m}$, and for $i\geq 2$,
\begin{align*}
    &\mu_{i}([j_1j_2\cdots j_i])=A(j_1,j_2)\cdots A(j_{i-1},j_i)\frac{t_{j_1,...,j_i;i}}{t_{\phi;i}},
\end{align*}
with
\begin{align*}
    t_{j_1,...,j_i;i}&=\prod_{k=1}^{i-1}f_k(j_{i-k},i),
\end{align*}
and for $1\leq k \leq i-1$, $f_{k}(j_{i-k},i)$ is defined recursively as
\begin{align*}
f_k(j_{i-k},i)&=\left[\sum_{b_1,b_2,...,b_{k-1}=0}^{m-1}A(j_{i-k},b_1)\prod_{e=1}^{k-2}A(b_e,b_{e+1})a_{b_{k-1}}\prod_{c=1}^{k-1}f_{c}(b_{k-c},i)\right]^{\frac{-d_{i,i-k}}{\sum_{\ell=0}^k d_{i,i-\ell}}},
\end{align*}
 and
\begin{equation}\label{eq 3.1-2}
    t_{\phi;i}=\sum_{j_1,...,j_i=0}^{m-1}A(j_1,j_2)A(j_2,j_3)\cdots A(j_{i-1},j_i)t_{j_1,...,j_i;i},
\end{equation}
and for $1\leq k \leq i-1$,
\begin{equation*}
    \mu_{i}([j_1j_2\cdots j_k])=\sum_{\ell=0}^{m-1}\mu_{i}([j_1j_2\cdots j_k \ell]),
\end{equation*}
and $\mu_\infty$ is Markov measure with the vector of initial
probabilities ${\bf p} = (\sum_{i=0}^{m-1}  ti)^{-1} (t_i)_{i=0}^{m-1}$, $t_{i}^{\frac{\gamma}{\alpha}}=\sum_{j=0}^{m-1}A(i,j)t_{j}$, and the matrix of transition probabilities
\begin{equation*}
    P(i,j)=\frac{t_j}{t_i^{\frac{\gamma}{\alpha}}} \mbox{ if } A(i, j) = 1, \mbox{ and otherwise }0.
\end{equation*}

Then, by the similar process as proof of \cite[Theorem 1.3 (2)]{ban2024hausdorffaffine}, the proof is complete.
\end{proof}

\begin{proof}[Proof of Theorem \ref{Thm: 2}]
\item[\bf 1.] Observe that for $1<\gamma<2$, and for any $\frac{|\beta|+|\delta|+2}{\gamma-1}\leq k \in \mathbb{N}$, 
\begin{equation*}
    \left\lfloor k+\beta \right\rfloor\geq \left\lfloor \frac{|\beta|+|\delta|+2}{\gamma-1}+\beta \right\rfloor\geq  \left\lfloor \frac{2}{\gamma-1} \right\rfloor\geq 1,
\end{equation*}
and if $\frac{|\beta|+|\delta|+2}{\gamma-1}+j\leq k < \frac{|\beta|+|\delta|+2}{\gamma-1}+j+1$ for some $j\in \mathbb{N}$, then
\begin{align*}
     \left\lfloor \gamma k+\delta \right\rfloor\geq&  \left\lfloor \gamma \frac{|\beta|+|\delta|+2}{\gamma-1}+\gamma j+\delta \right\rfloor\\
     =&\left\lfloor \frac{|\beta|+|\delta|+2}{\gamma-1}+\gamma j+1+|\beta|+|\delta|+\delta \right\rfloor+1\\
     \geq &\left\lfloor \frac{|\beta|+|\delta|+1}{\gamma-1}+j+1+\beta\right\rfloor+1\\
     \geq &\left\lfloor k+\beta \right\rfloor+1.
\end{align*}
This implies that $\left\lfloor k+\beta \right\rfloor\in S(1,\beta)$, $f(\left\lfloor k+\beta \right\rfloor)\in S(\gamma, \delta)$, and 
\begin{equation}\label{pfthm2-1-1}
    f(\left\lfloor k+\beta \right\rfloor)=\left\lfloor k'+\beta \right\rfloor\mbox{ for some positive integer }k'>k.
\end{equation}
Thus, there are at most $\frac{|\beta|+|\delta|+2}{\gamma-1}$ member in $A_i$ for all $i\geq 2$, i.e. $|A_i|\leq \frac{|\beta|+|\delta|+2}{\gamma-1}$. This implies that for $i\geq 2$, 
\begin{equation*}
    0\leq \frac{|A_i\cap [n_1, n_2]|}{|n_1-n_2|}\leq \frac{|A_i|}{|n_1-n_2|}\leq \frac{\frac{|\beta|+|\delta|+2}{\gamma-1}}{|n_1-n_2|}\to 0
\end{equation*}
as $|n_1-n_2|\to \infty$. Hence, $d_i=0$ for all $i\geq 2$. 

On the other hand, since for any $|\beta|+1\leq x\in \mathbb{N}$, $x=\left\lfloor k+\beta\right\rfloor$ for some $k\in\mathbb{N}$, we have 
\begin{equation}\label{pfthm2-1-3}
    x\in S(1,\beta).
\end{equation}
Thus, $|A_1|\leq |\mathbb{N}\setminus S(1,\beta)|\leq |\beta|+1$. This implies $d_1=0$.

Finally, by the definition of $A_\infty$ and (\ref{pfthm2-1-1}), we have
\begin{align*}
   A_\infty\subseteq\{x\in S(1,\beta)\setminus S(\gamma, \delta)\},
\end{align*}
and
\begin{align*}
   \{x\in S(1,\beta)\setminus S(\gamma, \delta): x=\left\lfloor k+\beta \right\rfloor\mbox{ with }k\geq \frac{|\beta|+|\delta|+2}{\gamma-1}\}\subseteq A_\infty.
\end{align*}
Hence, we obtain
\begin{equation}\label{pfthm2-1-2}
  \begin{aligned}
  &|[S(1, \beta)\setminus S(\gamma, \delta)]\cap [n_1, n_2]|-c \\
  \leq&  |A_\infty\cap [n_1, n_2]|\\
  \leq& |[S(1, \beta)\setminus S(\gamma, \delta)]\cap [n_1 , n_2]|\\
  \leq& |[n_1,n_2]\setminus S(\gamma, \delta)|,
\end{aligned}  
\end{equation}
where $c=\frac{|\beta|+|\delta|+2}{\gamma-1}$. 

Recall that (\ref{pfthm2-1-3}) gives that $S(1,\beta)=\mathbb{N}$ excepts a finite number, say $d\leq |\beta|+1$, of member of $\mathbb{N}$. This implies
\begin{equation}\label{pfthm2-1-4}
    |[S(1, \beta)\setminus S(\gamma, \delta)]\cap [n_1, n_2]|\geq |[n_1,n_2]\setminus S(\gamma, \delta)|-d
\end{equation}

Therefore, by (\ref{pfthm2-1-2}) and (\ref{pfthm2-1-4}), we have
\begin{align*}
   \frac{|[n_1,n_2]\setminus S(\gamma, \delta)|-c-d}{|n_1-n_2|}\leq  \frac{|A_\infty\cap [n_1, n_2]|}{|n_1-n_2|}\leq \frac{|[n_1,n_2]\setminus S(\gamma, \delta)|}{|n_1-n_2|}\to \frac{1}{\gamma}
\end{align*}
as $|n_1-n_2|\to\infty$. Thus, $d_\infty=\frac{1}{\gamma}$.

For the case $\gamma\geq 2$, the proof is similar to the case $1<\gamma <2$ by replacing $\frac{|\beta|+|\delta|+2}{\gamma-1}$ by $|\beta|+|\delta|+2$ in the beginning of the proof.

\item[\bf 2.] See Section \ref{appendix} for details.

\item[\bf 3.] Note that if $1<\tau\in \mathbb{Q}^c$ and $\eta\in \mathbb{R}$, then $x=\left\lfloor \tau k+\eta \right\rfloor$ is equivalent to 
\begin{equation*}
    0<\{x\tau'+\eta'\}\leq \tau',
\end{equation*}
where $\{n\}$ is the fractional part of $n$, and 
\begin{equation*}
    \tau'=\frac{1}{\tau}\mbox{ and }\eta'=\frac{1-\eta}{\tau}.
\end{equation*}

For the case $(\alpha,\beta, \gamma, \delta)\in \langle 3\rangle$, due to the fact that $\{\frac{1}{\alpha} x\}_{x\in \mathbb{N}}$ (resp. $\{\frac{1}{\gamma} x\}_{x\in \mathbb{N}}$) is uniformly distributed in $[0,1)$, we have that for any $x\in \mathbb{N}$, $x\in S(\alpha, \beta)$ has probability $\frac{1}{\alpha}$. On the other hand, since $\gamma=\frac{b}{a}\in \mathbb{Q}$ with $b>a>1$, $(a,b)=1$ gives that $S(\gamma, \delta)$ (resp. $S(\alpha, \beta)$) is an $a$-arithmetic progression (mod $b$) (see Section \ref{appendix} for the precise form), so we have 
\begin{equation*}
    d_1=\left(1-\frac{1}{\alpha}\right)\left(1-\frac{1}{\gamma}\right)=\frac{(\alpha-1)(\gamma-1)}{\alpha \gamma},
\end{equation*}
and for $i\geq 2$,
\begin{equation*}
   d_i=\left[\left(1-\frac{1}{\gamma}\right)\frac{1}{\alpha}\right]\left(\frac{1}{\alpha}\right)^{i-2} \left(1-\frac{1}{\alpha}\right)= \frac{\left( \alpha -1\right) \left( \gamma -1\right) }{\alpha ^{i}\gamma },
\end{equation*}
and 
\begin{equation*}
   d_\infty=\left[\left(1-\frac{1}{\gamma}\right)\frac{1}{\alpha}\right]\left(\frac{1}{\alpha}\right)^{\infty}=0.
\end{equation*}.

For the case $(\alpha,\beta, \gamma, \delta)\in \langle 4\rangle$, by the proof of \cite[Theorem 1]{harman2015primes}, if $\{1, \frac{1}{\alpha}, \frac{1}{\gamma}\}$ is linearly independent over $\mathbb{Q}$, then the fractional parts $\{\frac{1}{\alpha}x\}_{x\in\mathbb{N}}$ and $\{\frac{1}{\gamma}x\}_{x\in\mathbb{N}}$ are uniformly distributed in $[0,1)^2$. Thus, due to the uniformness, we have that for any $x\in \mathbb{N}$, $x\in S(\alpha, \beta)$ has probability $\frac{1}{\alpha}$ and $x\in S(\gamma, \delta)$ has probability $\frac{1}{\gamma}$, and this implies 
\begin{equation*}
    d_1=\left(1-\frac{1}{\alpha}\right)\left(1-\frac{1}{\gamma}\right)=\frac{(\alpha-1)(\gamma-1)}{\alpha \gamma},
\end{equation*}
and for $i\geq 2$,
\begin{equation*}
   d_i=\left[\left(1-\frac{1}{\gamma}\right)\frac{1}{\alpha}\right]\left(\frac{1}{\alpha}\right)^{i-2} \left(1-\frac{1}{\alpha}\right)= \frac{\left( \alpha -1\right) \left( \gamma -1\right) }{\alpha ^{i}\gamma },
\end{equation*}
and 
\begin{equation*}
   d_\infty=\left[\left(1-\frac{1}{\gamma}\right)\frac{1}{\alpha}\right]\left(\frac{1}{\alpha}\right)^{\infty}=0.
\end{equation*}

\item[\bf 4.] By \cite[Theorems 6 and 8]{harman2015primes} with (i) and (ii) respectively, we have that the set $S(\alpha,\beta)\cap S(\gamma,\delta)$ contains at most one element, i.e.,
\begin{equation}\label{pfthm2-4-1}
    |S(\alpha,\beta)\cap S(\gamma,\delta)|\leq 1.
\end{equation}
Since
\begin{equation*}
     A_1\cap [n_{1},n_{2}]=[n_1, n_2]\setminus [(S(\alpha, \beta)\cap [n_1, n_2])\cup (S(\gamma, \delta)\cap [n_1, n_2])],
\end{equation*}
we have
\begin{equation}\label{pfthm2-4-2}
 \begin{aligned}
    |A_1\cap [n_{1},n_{2}]|=&|[n_{1},n_{2}]|-|S(\alpha, \beta)\cap [n_1, n_2]|-|S(\gamma, \delta)\cap [n_1, n_2]|\\
    &+|S(\alpha, \beta)\cap S(\gamma, \delta)\cap [n_1, n_2]|.
\end{aligned}   
\end{equation}
Thus, by (\ref{pfthm2-4-1}) and (\ref{pfthm2-4-2}),
\begin{align*}
d_1=&\lim_{\left\vert n_{1}-n_{2}\right\vert \rightarrow \infty }\frac{\left\vert A_1\cap \lbrack n_{1},n_{2}]\right\vert }{\left\vert
n_{1}-n_{2}\right\vert }\\
=&\lim_{\left\vert n_{1}-n_{2}\right\vert \rightarrow \infty }1-\frac{|S(\alpha, \beta)\cap [n_1, n_2]|}{|n_1-n_2|}-\frac{|S(\gamma, \delta)\cap [n_1, n_2]|}{|n_1-n_2|}+\frac{c}{|n_1-n_2|}\\
=&1-\frac{1}{\alpha}-\frac{1}{\gamma},
\end{align*}
where $c$ is a constant equal to 0 or 1.

On the other hand, by (\ref{pfthm2-4-1}), we have
\begin{equation*}
    |A_2\cap [n_1, n_2]|=|S(\alpha, \beta)\cap [n_1, n_2]|\pm 2c.
\end{equation*}
Hence,
\begin{equation*}
    d_2=\lim_{\left\vert n_{1}-n_{2}\right\vert \rightarrow \infty }\frac{\left\vert A_2\cap \lbrack n_{1},n_{2}]\right\vert }{\left\vert
n_{1}-n_{2}\right\vert }=\lim_{\left\vert n_{1}-n_{2}\right\vert \rightarrow \infty }\frac{|S(\alpha, \beta)\cap [n_1, n_2]|\pm 2c}{\left\vert
n_{1}-n_{2}\right\vert }=\frac{1}{\alpha}.
\end{equation*}

For $i\geq 3$ and $i=\infty$, by (\ref{pfthm2-4-1}), $|A_i|=c$. This implies $d_i=0$. The proof is complete.
\end{proof}

\section{Appendix}\label{appendix}
 Let $1<\alpha<\gamma\in \mathbb{Q}$, we may assume that $\alpha=\frac{b}{a}, \gamma=\frac{d}{c}$ with $(a,b)=1, (c,d)=1$. Since 
\begin{align*}
     \left\lfloor \frac{b}{a}k\right\rfloor=b\ell+ \left\lfloor \frac{b}{a}h\right\rfloor~(0\leq h\leq a-1),\\
      \left\lfloor \frac{d}{c}k\right\rfloor=d\ell+ \left\lfloor \frac{d}{c}h\right\rfloor~(0\leq h\leq c-1),
\end{align*}
we have that $\mathbb{N}$ can be decomposed into the disjoint union of the following sets:
\begin{align*}
R_a^b(\beta):=&\left\{\left\lfloor \beta+\frac{bh}{a}\right\rfloor (\mbox{mod } b) : 0\leq h\leq a-1\right\}\mbox{ and }(R_a^b)^c:=\{0, ..., b-1\}\setminus R_a^b(\beta),\\
S_a^b:=& (b\mathbb{N}+R_a^b(\beta))\cap \mathbb{N},\\
    A_1=&\mathbb{N}\setminus (S_a^b\cup S_c^d),\\
    f:&S_a^b\to S_c^d,~ f(x):=\left\lfloor \frac{d}{c}\left(a\cdot \frac{x-\left\lfloor \frac{b}{a}h\right\rfloor}{b}+h\right)\right\rfloor,~\mbox{if }x\equiv \left\lfloor \frac{b}{a}h\right\rfloor(\mbox{mod } b)~(0\leq h\leq a-1),\\
    A_2=&\left\{x\in S_a^b \setminus S_c^d:f(x)\in S_c^d \setminus S_a^b\right\},\\
    A_i=&\left\{x\in S_a^b \setminus S_c^d:f^{j}(x)\in S_a^b \cap S_c^d ~(\forall 1\leq j\leq i-2)\mbox{ and }f^{i-1}(x)\in S_c^d\setminus S_a^b \right\}~(i\geq 3),\\
    A_{\infty}=&\left\{x\in S_a^b \setminus S_c^d:f^{j}(x)\in S_a^b \cap S_c^d ,~\forall  j\geq 1 \right\}.
\end{align*}

\begin{lemma}\label{lma 1}
    Let $a,b\in\mathbb{N}$ and $0\leq i\leq a-1, 0\leq j\leq b-1$. We have
    \begin{equation*}
          g(a,b,i,j):=\lim_{|n_1-n_2|\to\infty}\frac{|a\mathbb{N}+i\cap b\mathbb{N}+j\cap [n_1,n_2]|}{|n_1-n_2|}=\left\{\begin{array}{cc}
         \frac{(a,b)}{ab},~&\mbox{if }(a,b)|(i-j),  \\
         0,~&\mbox{if }(a,b)\nmid(i-j).
    \end{array}\right.
    \end{equation*}  
\end{lemma}

\begin{proof}
    If $(a,b)|(i-j)$, then there exist $t_1,t_2\in\mathbb{Z}$ such that
    \begin{equation}\label{eq l 1}
        at_1-bt_2=|i-j|.
    \end{equation}
    One can choose $t_1,t_2$ to be the smallest nonnegative integer satisfying (\ref{eq l 1}). Then, $s_1=t_1+\frac{b}{(a,b)} m$ and $s_2=t_2+\frac{a}{(a,b)} m$ ($m> 0$) are positive integers satisfying 
    \begin{equation*}
        as_1-bs_2=|i-j|.
    \end{equation*}
    Thus,
    \begin{equation*}
        a\mathbb{N}+i\cap b\mathbb{N}+j=\{as_1+i=a(t_1+\frac{b}{(a,b)}m)+i:  m>0\}.
    \end{equation*}
    This implies that
    \begin{equation*}
         \lim_{|n_1-n_2|\to\infty}\frac{|a\mathbb{N}+i\cap b\mathbb{N}+j\cap [n_1,n_2]|}{|n_1-n_2|}=\frac{(a,b)}{ab}.
    \end{equation*}

    On the other hand, if $(a,b)\nmid(i-j)$, then $a\mathbb{N}+i\cap b\mathbb{N}+j=\emptyset$. Thus,
     \begin{equation*}
         \lim_{|n_1-n_2|\to\infty}\frac{|a\mathbb{N}+i\cap b\mathbb{N}+j\cap [n_1,n_2]|}{|n_1-n_2|}=0.
    \end{equation*}
    The proof is complete.
\end{proof}

By Lemma \ref{lma 1}, we have
\begin{equation}\label{eq case2}
\begin{aligned}
    d_1=&\lim_{|n_1-n_2|\to\infty}\frac{|A_1\cap [n_1,n_2]|}{|n_1-n_2|}\\
    =&\sum_{i\in (R_a^b)^c, j\in (R_c^d)^c}\lim_{|n_1-n_2|\to\infty}\frac{|b\mathbb{N}+i\cap d\mathbb{N}+j\cap [n_1,n_2]|}{|n_1-n_2|}\\
    =&\sum_{i\in (R_a^b)^c, j\in (R_c^d)^c}g(b,d,i,j),
     \end{aligned}
     \end{equation}
and for $i\geq 2$,
\begin{equation}\label{eq case2-1}
\begin{aligned}
     d_i=&\lim_{|n_1-n_2|\to\infty}\frac{|A_i\cap [n_1,n_2]|}{|n_1-n_2|}\\
     =&\sum_{i\in R_a^b, j\in (R_c^d)^c}g(b,d,i,j)\left(\frac{\sum_{i\in R_a^b, j\in R_c^d}g(b,d,i,j)}{\sum_{ j\in R_c^d}g(1,d,0,j)}\right)^{i-2} \frac{\sum_{i\in (R_a^b)^c, j\in R_c^d}g(b,d,i,j)}{\sum_{ j\in R_c^d}g(1,d,0,j)},
     \end{aligned}
     \end{equation}
 and
\begin{equation}\label{eq case2-2}
\begin{aligned}
     d_\infty=&\lim_{|n_1-n_2|\to\infty}\frac{|A_\infty\cap [n_1,n_2]|}{|n_1-n_2|}\\
     =&\sum_{i\in R_a^b, j\in (R_c^d)^c}g(b,d,i,j)\left(\frac{\sum_{i\in R_a^b, j\in R_c^d}g(b,d,i,j)}{\sum_{ j\in R_c^d}g(1,d,0,j)}\right)^{\infty}. 
     \end{aligned}
     \end{equation}

\begin{proof}[Proof of Corollary \ref{Cor: 1}]
    The proof is directly obtained by (\ref{eq case2}), (\ref{eq case2-1}), (\ref{eq case2-2}) and Lemma \ref{lma 1}. 
    
    More precisely, in $\langle 7\rangle$, $\alpha=1=\frac{1}{1}$ and $\gamma=\frac{\gamma}{1}$, then $a=1, b=1, c=1, d=\gamma$, $R_a^b(\beta)=\{0\}$, $R_c^d(\delta)=\{ \delta (\mbox{mod } \gamma) \}$, $(R_a^b)^c=\emptyset$, $(R_c^d)^c=\{0, ..., \gamma-1\}\setminus\{ \delta (\mbox{mod } \gamma) \}$. Then, since $(R_a^b)^c=\emptyset$ implies 
    \begin{equation*}
        \sum_{i\in (R_a^b)^c, j\in (R_c^d)^c}g(b,d,i,j)=0,
    \end{equation*}
    and
    \begin{equation}\label{pf 6-1}
   \sum_{i\in (R_a^b)^c, j\in R_c^d}g(b,d,i,j)=0,     
    \end{equation}
    we have $d_i=0$ for all $i\in \mathbb{N}$. On the other hand, since (\ref{pf 6-1}) implies 
    \begin{equation*}
       \sum_{i\in R_a^b, j\in R_c^d}g(b,d,i,j)=\sum_{ j\in R_c^d}g(1,d,0,j) ,
    \end{equation*}
    and due to the fact that
    \begin{equation*}
        \sum_{i\in R_a^b, j\in (R_c^d)^c}g(b,d,i,j)=1-\sum_{i\in R_a^b, j\in R_c^d}g(b,d,i,j),
    \end{equation*}
    by Lemma \ref{lma 1}, we have $d_\infty=1-\frac{1}{\gamma}$. Thus, ${\bf d}=(0, 0, \underline{0}, 1-\frac{1}{\gamma}) $. 
    
    In $\langle 8 \rangle$, $\alpha=\frac{\alpha}{1}$ and $\gamma=\frac{\gamma}{1}$, then $a=1, b=\alpha, c=1, d=\gamma$, $R_a^b(\beta)=\{ \beta (\mbox{mod } \alpha) \}$, $R_c^d(\delta)=\{ \delta (\mbox{mod } \gamma) \}$, $(R_a^b)^c=\{0, ..., \alpha-1\}\setminus\{ \beta (\mbox{mod } \alpha) \}$, $(R_c^d)^c=\{0, ..., \gamma-1\}\setminus\{ \delta (\mbox{mod } \gamma) \}$. Thus, since Lemma \ref{lma 1} and $(\alpha,\gamma)\nmid (\beta-\delta)$ implies 
    \begin{equation}\label{pf 7-1}
        \sum_{i\in R_a^b, j\in R_c^d}g(b,d,i,j)=0,
    \end{equation}
    we have 
    \begin{align*}
        d_1=&\sum_{i\in (R_a^b)^c, j\in (R_c^d)^c}g(b,d,i,j)\\
        =&1-\sum_{i\in R_a^b}\sum_{j=1}^d g(b,d,i,j)-\sum_{j\in R_c^d}\sum_{i=1}^b g(b,d,i,j)+\sum_{i\in R_a^b, j\in R_c^d}g(b,d,i,j)\\
        =&1-\sum_{i\in R_a^b} g(b,1,i,0)-\sum_{j\in R_c^d} g(1,d,0,j)+\sum_{i\in R_a^b, j\in R_c^d}g(b,d,i,j)\\
        =& 1-\frac{1}{\alpha}-\frac{1}{\gamma},
    \end{align*}
    and since (\ref{pf 7-1}) implies 
    \[\sum_{i\in (R_a^b)^c, j\in R_c^d}g(b,d,i,j)=\sum_{j\in R_c^d}g(1,d,0,j),\]
    we have
    \begin{align*}
        d_2=&\sum_{i\in R_a^b, j\in (R_c^d)^c}g(b,d,i,j) \frac{\sum_{i\in (R_a^b)^c, j\in R_c^d}g(b,d,i,j)}{\sum_{ j\in R_c^d}g(1,d,0,j)}\\
        =&\sum_{i\in R_a^b, j\in (R_c^d)^c}g(b,d,i,j) \\
        =&\sum_{i\in R_a^b}g(b,1,i,0)\\
        =&\frac{1}{\alpha},
    \end{align*}
    and (\ref{pf 7-1}) implies $d_i=0$ for all $i\geq 3$ and $i=\infty$. Hence, $\mathbf{d}=(1-\frac{1}{\alpha }-\frac{1}{\gamma },\frac{1}{\alpha },\underline{0},0)$. 
    
    In $\langle 9 \rangle$, $\alpha=\frac{\alpha}{1}$ and $\gamma=\frac{\gamma}{1}$, then $a=1, b=\alpha, c=1, d=\gamma$, $R_a^b(\beta)=\{ \beta (\mbox{mod } \alpha) \}$, $R_c^d(\delta)=\{ \delta (\mbox{mod } \gamma) \}$, $(R_a^b)^c=\{0, ..., \alpha-1\}\setminus\{ \beta (\mbox{mod } \alpha) \}$, $(R_c^d)^c=\{0, ..., \gamma-1\}\setminus\{ \delta (\mbox{mod } \gamma) \}$. Since $(\alpha,\gamma)\mid (\beta-\delta)$ implies $(\alpha,\gamma)\mid (\beta(\mbox{mod }\alpha)- \delta(\mbox{mod }\gamma))$, by Lemma \ref{lma 1}, we have 
    \begin{equation*}
        \sum_{i\in R_a^b, j\in R_c^d}g(b,d,i,j)=\frac{(\alpha,\gamma)}{\alpha \gamma}=\frac{1}{(\alpha
,\gamma )\alpha _{1}\gamma _{1}}.
    \end{equation*}
    Hence, 
    \begin{align*}
             d_1=&1-\sum_{i\in R_a^b} g(b,1,i,0)-\sum_{j\in R_c^d} g(1,d,0,j)+\sum_{i\in R_a^b, j\in R_c^d}g(b,d,i,j)\\
        =& 1-\frac{1}{\alpha}-\frac{1}{\gamma}+\frac{1}{(\alpha
,\gamma )\alpha _{1}\gamma _{1}}.
    \end{align*}
Note that for $i\in \{0, ..., \alpha-1\}$ and $j\in \{0, ..., \alpha \gamma_1-1\}$, if $\alpha \mid (i-j)$, then $\alpha \nmid (i'-j)$ for all $i'\in \{0, ..., \alpha -1\}\setminus \{i\}$. We have that $(\alpha,\gamma)\mid (\beta(\mbox{mod }\alpha)- \delta(\mbox{mod }\gamma))$, $\alpha=(\alpha,\gamma)$ and Lemma \ref{lma 1} implies 
\begin{equation*}\label{pf 9-2}
    \sum_{i\in (R_a^b)^c, j\in R_c^d}g(b,d,i,j)=0.
\end{equation*}
Thus, $d_i=0$ for all $i\geq 2$. Note that since $\alpha_1=1$ implies $\gamma= \alpha \gamma_1$, we have
    \begin{equation*}
        (R_c^d)^c=\{0, ..., \alpha\gamma_1-1\}\setminus\{ \delta (\mbox{mod } \alpha\gamma_1) \}.
    \end{equation*}
Now, due to the fact that $\alpha\mid (i-j) \mbox{ if and only if } i\equiv j (\mbox{mod }\alpha)$, we have that for all $i \in \{0, ..., \alpha-1\}, j\in \{0, ..., \alpha \gamma_1-1\}$,
\begin{equation*}
    \alpha\mid (i-j) \mbox{ if and only if } j=i+\ell \alpha \mbox{ for some }0\leq \ell \leq \gamma_1-1.
\end{equation*}
Thus, $\delta (\mbox{mod } \alpha\gamma_1)=\beta(\mbox{mod }\alpha)+\ell \alpha $ for some $0\leq \ell \leq \gamma_1-1$. Hence, there are exactly $\gamma_1-1$ many members $j\in (R_c^d)^c$ such that $\alpha \mid (\beta(\mbox{mod }\alpha)-j)$. Combining this with Lemma \ref{lma 1}, we have
\begin{equation}\label{pf 9-1}
    \sum_{i\in R_a^b, j\in (R_c^d)^c}g(b,d,i,j)= (\gamma_1-1)\frac{(\alpha,\gamma)}{\alpha \gamma}=\frac{\gamma _{1}-1}{\alpha
\gamma_{1}}.
\end{equation}
On the other hand, Lemma \ref{lma 1} and $\alpha_1=1$ implies
\begin{equation}\label{pf 9-3}
    \frac{\sum_{i\in R_a^b, j\in R_c^d}g(b,d,i,j)}{\sum_{ j\in R_c^d}g(1,d,0,j)}=\frac{\frac{(\alpha,\gamma)}{\alpha \gamma}}{ \frac{1}{\gamma}}=1.
\end{equation}
Thus, by (\ref{pf 9-1}) and (\ref{pf 9-3}), we obtain $d_\infty=\frac{\gamma _{1}-1}{\alpha
\gamma_{1}}$. Hence,    
\[
\mathbf{d}=\left( 1-\frac{1}{\alpha }-\frac{1}{\gamma }+\frac{1}{(\alpha
,\gamma )\alpha _{1}\gamma _{1}},0,\underline{0},\frac{\gamma _{1}-1}{\alpha
\gamma_{1}}\right).
\]

In $\langle 10 \rangle$, $d_1$ is the same as in $\langle 9 \rangle$. By Lemma \ref{lma 1} , we have 
\begin{equation}\label{pf 10-1}
    \frac{\sum_{i\in R_a^b, j\in R_c^d}g(b,d,i,j)}{\sum_{ j\in R_c^d}g(1,d,0,j)}=\frac{\frac{(\alpha,\gamma)}{\alpha \gamma}}{ \frac{1}{\gamma}}=\frac{1}{\alpha_1}.
\end{equation}
Since $\alpha_1>1$, we have $d_\infty=0$. By the similar discussion as in (\ref{pf 9-1}), we obtain
\begin{equation}\label{pf 10-2}
    \sum_{i\in R_a^b, j\in (R_c^d)^c}g(b,d,i,j)=\frac{\gamma _{1}-1}{\alpha
\gamma_{1}},
\end{equation}
and 
\begin{equation}\label{pf 10-3}
    \frac{\sum_{i\in (R_a^b)^c, j\in R_c^d}g(b,d,i,j)}{\sum_{ j\in R_c^d}g(1,d,0,j)}=\frac{(\alpha_1-1)\frac{(\alpha,\gamma)}{\alpha \gamma}}{ \frac{1}{\gamma}}=\frac{\alpha_1-1}{\alpha_1}.
\end{equation}
Hence by (\ref{pf 10-2}) and (\ref{pf 10-3}), we have that for $i\geq 2$,
\begin{equation*}
    d_i=\frac{(\gamma _{1}-1)(\alpha_{1}-1)}{\gamma_{1}\alpha\alpha_{1}^{i-1}}.
\end{equation*}
Thus,
\[
\mathbf{d}=\left( 1-\frac{1}{\alpha }-\frac{1}{\gamma }+\frac{1}{(\alpha
,\gamma )\alpha _{1}\gamma _{1}},\frac{(\gamma _{1}-1)(\alpha _{1}-1)}{\gamma _{1}\alpha \alpha _{1}},\underline{d},0\right) \text{,} 
\]
where 
\[
\underline{d}=\left( \frac{(\gamma _{1}-1)(\alpha_{1}-1)}{\gamma_{1}\alpha\alpha_{1}^{i-1}}\right)_{i\geq 3}
\text{.} 
\]
 The proof is complete.
\end{proof}

\bibliographystyle{amsplain}
\bibliography{ban}
\end{document}